\documentclass[12pt]{amsart}
\usepackage{mathtools}
\usepackage{titlesec,caption, subcaption}
\usepackage{tikz-cd}
\usepackage{amsmath,amsthm,amsfonts,amscd,amssymb,amsopn,enumerate,enumitem}
\usepackage{mathrsfs,xcolor,hyperref,stmaryrd}
\usepackage{wasysym}
\usepackage[all]{xy}
\usepackage{fullpage}
\usepackage{verbatim}
\usepackage{colonequals}
\usepackage[new]{old-arrows}
\usepackage{blindtext}

\titleformat{name=\section}{}{\thetitle.}{0.8em}{\centering\scshape}
\titleformat{name=\subsection}[runin]{}{\thetitle.}{0.8em}{\bfseries}[.]
\titleformat{name=\subsubsection}[runin]{}{\thetitle.}{0.8em}{\bfseries}[.]

\DeclareRobustCommand\longtwoheadrightarrow
     {\relbar\joinrel\twoheadrightarrow}

\theoremstyle{plain}
\newtheorem{thm}{Theorem}[section]
\newtheorem*{thm*}{Theorem}

\newtheorem*{conj*}{Conjecture}

\newtheorem*{fact*}{Fact}
\newtheorem*{prop*}{Proposition}

\theoremstyle{definition}
\newtheorem{defn}[thm]{Definition}
\newtheorem*{defn*}{Definition}
\newtheorem*{theorem*}{Theorem}

\newtheorem{rem}[thm]{Remark}

\newtheorem{prop}[thm]{Proposition}
\newtheorem{theorem}[thm]{Theorem}
\newtheorem{lemma}[thm]{Lemma}
\newtheorem{cor}[thm]{Corollary}

\theoremstyle{remark}

\newcommand{\loc}[1]{#1_{\mathfrak{p}}}

\DeclareMathOperator{\Ext}{Ext}
\DeclareMathOperator{\Hom}{Hom}
\DeclareMathOperator{\res}{res}
\DeclareMathOperator{\identity}{id}
\DeclareMathOperator{\im}{im}
\DeclareMathOperator{\spec}{Spec}

\title{Arithmetic Invariant Theory of Reductive Groups}

\author{Yidi Wang}

\subjclass[2020]{13A50, 14L30, 20G35}
\keywords{Invariant theory, reductive groups, group schemes, Cohen--Macaulay rings.}

\begin{document}

\begin{abstract}
    In this manuscript, we define the notion of linearly reductive groups over commutative unital rings and study the Cohen--Macaulay property of the ring of invariants under rational actions of a linearly reductive group. Moreover, we study the equivalence of different notions of reductivity over regular rings of Krull dimension two by studying these properties locally. 
\end{abstract}

\maketitle

\section{Introduction}

Invariant theory over fields, especially over fields of characteristic zero, was much studied in the previous century. 
In general, one studies the structural properties of the ring of invariants~$S^G$ of a group~$G$ acting rationally on a ring~$S$. 
One may refer to \cite{cit} or \cite{dolgachev_2003} for the classical theory. 
In recent years, Almuhaimeed and Mundelius generalized some of the classical results to certain rings for finite groups $G$ in their Ph.D. theses (see \cite{AA2018} and \cite{mundelius}).
The ring of invariants~$S^G$ also has rich properties when the group~$G$ is a reductive algebraic group. 
There are in fact several different notions of reductive groups whose relationships have been studied in the literature: Miyata--Nagata proved that the notions of linearly reductive and reductive are equivalent over fields of characteristic zero and that linearly reductive implies reductive over fields of positive characteristic. 
They also showed that geometrically reductive implies reductive (\cite{Nagata_Miyata}). 
Haboush later showed the converse, thereby giving an affirmative answer to Mumford's Conjecture (\cite{haboush}). 
For different notions of reductive groups over rings, following Haboush's setting, Seshadri generalized geometric reductivity to smooth algebraic group schemes with connected geometric fibres over Noetherian rings and proved that reductive group schemes are geometrically reductive (\cite{seshadri}). 
The converse is unknown in general. In \cite{Franjou_VanderKallen}, Franjou--van der Kallen formulated power reductivity over arbitrary bases which is compatible with the setup in Mumford's Conjecture, and proved that reductive group schemes are power reductive and that power reductive group schemes are geometrically reductive. 
The converse is also unknown in general. 
One goal of this paper is to study the unknown directions over certain classes of rings. 

In  Section \ref{sec_lrg}, we introduce the notion of linearly reductive groups, and then compare different notions of reductive groups over certain commutative unital rings. 
A similar notion was also studied in \cite{hashimoto}, but we focus differently on the properties of representations of linearly reductive groups in analogy to what is known over fields. Over fields, it is well-known that~$G$ is linearly reductive if and only if there exists a \emph{Reynolds operator} for all~$G$-modules~$V$, or equivalently,  its invariant module~$V^G$ splits in~$V$.
We extend this result to certain commutative rings:
\begin{thm*}[\ref{theorem:reynolds}]
    Let~$G$ be a smooth affine algebraic group scheme over a commutative ring with connected geometric fibres. If a functorial Reynolds operator exists for~$G$, then~$G$ is linearly reductive. The converse holds if we further assume~$k$ is a Dedekind domain.
\end{thm*}

The Cohen--Macaulay property measures the local equidimensionality of a ring and therefore has rich applications in algebraic geometry. 
Hochster--Roberts showed that when~$G$ is linearly reductive over a field~$F$,~$F[V]^G$ is Cohen--Macaulay for all rational representations~$V$ (\cite{Hochster-Robert}). 
Over~$\mathbb{Z}$, such a property was also studied in \cite{AA2018} for finite groups. 
Hochster--Huneke also proved that~$F[X]^G$ is Cohen--Macaulay for all smooth~$G$-varieties~$X$ (\cite{hochster-huneke}). 
We generalize Hochster--Huneke's result to Dedekind domains in Section~\ref{subsection:cm}, using our new definition of linearly reductive groups.

\begin{thm*}[\ref{CM}]
    Let~$k$ be a Dedekind domain and let~$G$ be a linearly reductive group over~$k$. Let~$X \rightarrow k$ be a smooth affine algebraic~$G$-scheme with connected generic fibre~$X_{\operatorname{Frac}(k)}$. Let ~$k[X]$ be the coordinate ring of~$X$. Then~$k[X]^G$ is Cohen--Macaulay.
\end{thm*}

In Section \ref{sec_rg}, we compare our linear reductivity to other notions of reductivity and show the converse of Seshadri and Franjou--van der Kallen's results for regular rings of Krull dimension at most two.

\begin{thm*}[Proposition \ref{PR_GR}, Corollary \ref{GR_implies_PR}, Proposition \ref{PR_deg_one_property}]
Let~$G$ be a flat affine algebraic group scheme over a commutative unital ring~$k$.
    \begin{enumerate}[label = (\alph*)]
        \item Over a commutative ring~$k$, a linearly reductive group ~$G$ is reductive. The converse is true if~$char(\kappa(\mathfrak{p})) = 0$ for all~$\mathfrak{p} \in \operatorname{Spec}(k)$, where~$\kappa(\mathfrak{p})$ denotes the residue field at~$\mathfrak{p}.$
        \item We further assume that~$G$ is smooth with connected geometric fibres. Over a regular ring of Krull dimension at most two, the notions of reductive groups, power reductive groups and geometrically reductive groups coincide.
    \end{enumerate} 
\end{thm*}

\noindent Note that the second part of this result was previously shown in \cite{Alper_GR}, using the language of stacks. In our approach, the key step by to obtaining the second part is by a different argument to show that geometric reductivity is a local property (Theorem \ref{global_to_local}), which does not rely on the knowledge of stacks.

\subsection*{Acknowledgements} This manuscript contains results from the author's doctoral dissertation under the supervision of Julia Hartmann at the University of Pennsylvania. The author thanks her for the valuable mathematical discussions as well as her patient advising and generous support. The author would also like to thank David Harbater for helpful discussions, Daniel Krashen for pointing out a problem in the earlier version of this manuscript, and Man Cheung Tsui for carefully checking Theorem ~\ref{global_to_local}. The author was partially supported by the NSF grant DMS-2102987.

\section{Preliminaries for Representations of affine group schemes}
Throughout this manuscript, we work with unital rings. In this section, we recall some definitions and properties of representations of affine group schemes over rings from the literature. Throughout this section, let~$k$ be a commutative ring unless specified.

\subsection{Representations of algebraic groups}\label{sec_affine_group_scheme}

Let~$G$ be an affine group scheme over~$k$. Let~$V$ be a~$k$-module. A \emph{$G$-module} structure on~$V$ is a natural transformation~$r_V: G \rightarrow \textnormal{GL}(V)$, i.e.,~$G(A)$ acts~$A$-linearly on~$V\otimes_k A$, for all~$k$-algebras~$A$. A~$k$-submodule~$W$ of~$V$ is~$G$-\emph{stable} (or a~$G$-\emph{submodule} of~$V$) if there is a natural transformation~$r_W: G \rightarrow \textnormal{GL}(W).$

There is also a Hopf-algebraic approach of describing~$G$-modules. Let~$G$ be an affine group scheme over~$k$. A~$G$-module structure on a~$k$-module on~$V$ corresponds to a~$k$-linear map~$\Delta_V: V  \rightarrow V \otimes k[G]$, called the \emph{comodule map} of a~$G$-module~$V$, such that the following two diagrams commute.
\begin{center}
    \begin{tikzcd}[column sep = large]
        &V \arrow{r}{\Delta_V}\arrow{d}{\Delta_V} &V \otimes k[G]\arrow{d}{\operatorname{id}_V\otimes \Delta_{G}} &V \arrow{r}{\Delta_V}\arrow[equal]{d} &V \otimes k[G] \arrow{d}{\operatorname{id}\otimes \epsilon}\\
        &V \otimes k[G] \arrow{r}{\Delta_V\otimes \operatorname{id}_{k[G]}} &V \otimes k[G] \otimes k[G], &V \arrow{r}{} &V \otimes k
    \end{tikzcd}
\end{center} Here~$\Delta_{G}$ denotes the comultiplication map and~$\epsilon$ denotes the antipode map. One may refer to \cite[Section 2.8]{Jantzen} for details.

A homomorphism of~$k$-modules~$\varphi \colon V \rightarrow W$ is a \emph{homomorphism of~$G$-modules} if the following diagram is commutative.
    \begin{equation*}\label{G-mod_homo}
        \begin{tikzcd}
            &V \arrow{r}{\Delta_V}\arrow{d}{\varphi} &V \otimes k[G]\arrow{d}{\varphi \otimes \operatorname{id}_{k[G]}}\\
            &W \arrow{r}{\Delta_W} &W \otimes k[G]
        \end{tikzcd}
    \end{equation*}

    \begin{rem}\label{alg_operation}
        The~$G$-module structure is preserved under taking tensor product, direct sum, symmetric and exterior powers of one or several~$G$-modules. If a~$G$-module~$M$ is finitely generated and projective as a $k$-module, the module~$\Hom_k(M, N)$ is a~$G$-module in a natural way. In particular, the dual module~$M^{\ast}$ is naturally a~$G$-module. See for example, \cite[Section 2.7]{Jantzen}.
    \end{rem}

    \begin{lemma}~\label{dual_proj}
    Let~$s \colon N \to M$ be a homomorphism of~$G$-modules over a commutative ring~$k$. Suppose both~$M$ and~$N$ are finitely generated and projective as~$k$-modules. Then the induced map~$s^{\ast} \colon M^{\ast} \rightarrow N^{\ast}$ is~$G$-equivariant.
    \end{lemma}

\begin{proof}
    For any~$k$-algebra~$A$,~$M^{\ast} \otimes_k A \cong (M\otimes_k A)^{\ast}$ by \cite[Chapter II, §5.4, Prop.8]{bourbaki1998algebra}. The~$G$-module structure on~$M^{\ast}$ is given by~$g \cdot f(m) \coloneqq f(g^{-1}m)$ for all~$m \in M \otimes_k A$ and~$g \in G(A)$. To check that~$s^{\ast}$ defines a~$G$-module homomorphism, one checks that~$g \cdot s^{\ast}(f)(n) = s^{\ast}(g\cdot f)(n)$.
\end{proof}

    A~$G$-module~$V$ is \emph{locally finite} if for all~$v \in V$, there exists a~$G$-stable submodule~$W$ of $V$ that is finitely generated over $k$ such that~$v \in W$.

    \begin{lemma}\label{locally_finiteness}
        Any~$G$-module~$V$ is an inductive limit of ~$G$-modules that are finitely generated over $k$ if~$k$ is Noetherian and~$G$ is flat. In particular,~$V$ is locally finite. 
    \end{lemma}
    \begin{proof}
        It follows from \cite[Proposition 2, Corollaire]{serre68}. 
    \end{proof}

    \begin{rem}\label{lemma:projective_seshadri}
        In particular, when~$G$ is smooth and has connected geometric fibres,~$k[G]$ is projective (see \cite[Lemma 1]{seshadri}), and therefore any~$G$-module~$V$ is locally finite.
    \end{rem}

    One can also define group actions on a scheme.  A \emph{group action} of~$G$ on a~$k$-scheme~$X$ is a natural transformation~$m\colon G \times X \rightarrow X$ such that, for all~$k$-algebras~$A$,
    \begin{equation*}
         m(A): G(A) \times X(A) \rightarrow X(A)
    \end{equation*} defines a group action. In particular, if~$X$ is affine, the induced coaction 
    \begin{equation*}
        \alpha^{\ast}: k[X]  \longrightarrow k[G] \otimes k[X]
    \end{equation*} is a comodule map for~$k[X]$. A scheme~$X$ with a~$G$-action is called a \emph{$G$-scheme}.

     Let~$G$ be a flat affine algebraic group scheme over a Noetherian ring~$k$. We say that $(G, \operatorname{Spec}(k))$ has the \emph{$G$-equivariant resolution property} if for every finitely generated~$G$-module~$M$, there is a ~$G$-module~$V$ that is finitely generated and projective over $k$ and a~$G$-equivariant epimorphism~$V \twoheadrightarrow M$. For a general definition, see \cite[Definition 2.1]{THOMASON1987}.

    \begin{lemma}\label{lemma:thomason_equivariant}
        If~$k$ is a regular ring of Krull dimension at most two, and~$G$ is a smooth affine group scheme over~$k$ with connected geometric fibres, then~$(G, \operatorname{Spec}(k))$ has the~$G$-equivariant resolution property.
    \end{lemma}

    \begin{proof}
        See \cite[Lemma 2.5]{THOMASON1987}.
    \end{proof}

    \begin{rem}
        Note that in various situations the local finiteness and the~$G$-equivariant resolution property allow us to reduce to working with~$G$-modules that are finitely generated and projective over $k$, for which more algebraic operations naturally inherit the~$G$-module structure. See Remark \ref{alg_operation}.
    \end{rem}

\subsection{Restriction and extension of scalars} 
Let~$f: R \rightarrow S$ be a homomorphism of commutative rings. If~$M$ is an~$S$-module, then~$M$ is also an~$R$-module, denoted~$M_R$, where~$r\cdot m:= f(r)\cdot s$. This is called the \emph{restriction of scalars}. If~$N$ is an~$R$-module. then we obtain an~$S$-module by applying the tensor functor~$(\cdot)\otimes_R S$. This is called the \emph{extension of scalars}. 

Let~$G$ be an affine algebraic group scheme over~$R$, and let~$G_S\coloneqq G \times_R S$. Let~$M$ be a~$G_S$-module. Then~$M_R$ is naturally a~$G$-module. Let~$V$ be a~$G$-module. Then~$V \otimes_R S$ is naturally a~$G_S$-module. See \cite[10.1]{Jantzen}.

\begin{prop}\label{base_change_representation}

    Let notation be as above. If~$f$ is a surjective homomorphism of commutative rings or~$S$ is the localization of~$R$ at a prime ideal, then~$M_R \otimes_R S \cong M$ as~$G_S$-modules.
    
\end{prop}

\begin{proof}
    First note that~$M_R \otimes_R S \cong (M \otimes_S S) \otimes_R S  \cong M \otimes_S (S \otimes_R S).$
    The first isomorphism makes sense since~$M \otimes_S S \cong M$ as both~$S$-modules and~$R$-modules. The second isomorphism follows from the associativity of tensor product of modules. See \cite[Proposition 2.57]{Rotman}. Then~$M_R \otimes_R S \cong M$ if~$S \otimes_R S = S$.

    If~$f: R \twoheadrightarrow S$ is surjective, then let~$I$ denote the kernel of~$f$. Then we have the following exact sequence of~$R$-algebras:~$0 \longrightarrow I \longrightarrow R \longrightarrow S \longrightarrow 0$.
    By tensoring with~$S$ over~$R$, we obtain
        ~$I\otimes_R S \overset{j}{\longrightarrow} R \otimes_R S \longrightarrow S \otimes_R S\longrightarrow 0$.
     However,~$I \otimes_R S \cong I \otimes_R R/I$ is trivial. Therefore,~$S \otimes_R S \cong S/\operatorname{im}(j) \cong S$. If~$S$ is a localization of~$R$ at a prime ideal~$\mathfrak{p}$, then by \cite[Proposition 3.7]{AM94},~$R_{\mathfrak{p}} \otimes_R R_{\mathfrak{p}} \cong (R \otimes_R R)_{\mathfrak{p}} \cong R_{\mathfrak{p}}$.

    It remains to check that in both of these cases, the isomorphism~$M_R \otimes_R S \cong M$ preserves the~$G_S$-module structure. Let~$A$ be an~$S$-algebra, so~$A$ is also an~$R$-algebra by restriction of scalars. The action of~$G_S(A)$ on~$M \otimes_S A$ is compatible with the action of~$G_S(A \otimes_R S)$ on~$(M_R \otimes_R S) \otimes_S A$, since~$A \otimes_R S \cong A \otimes_S (S \otimes_R S) \cong A$ in the above two cases.
\end{proof}

\subsection{A functorial definition of invariants}

    Let~$G$ be an affine group scheme and~$V$ a~$G$-module. The \emph{invariants} are defined as 
    \begin{equation*}
        V^G = \{v \in V \mid g(v \otimes 1) = v \otimes 1 \textnormal{ for all } g \in G(A), \textnormal{ for all } k \textnormal{-algebra } A\}.
    \end{equation*} Note that~$V^G$ forms a~$G$-module.
This definition also coincides with invariants defined from coactions (see for example \cite[2.10]{Jantzen}):
    \begin{equation*}
        V^G = \{v \in V\mid \Delta_V(v) = v \otimes 1\},
    \end{equation*} where~$\Delta_V$ denotes the comodule map for~$V$.

Note that here~$V^G$ is described as the kernel of the map~$\Delta_V - \operatorname{id}_V \otimes 1$ and therefore, taking invariants commutes with flat base change. That is, ~$M^G \otimes_k R \cong (M \otimes_k R)^{G_R}$ as~$G_R$-modules if ~$k \rightarrow R$ is a flat extension. See for example \cite[2.10]{Jantzen} or \cite[Lemma 2]{seshadri}.

\section{Linearly reductive groups}\label{sec_lrg}
In this section we introduce the notion of linearly reductive group schemes over a commutative ring~$k$. We will also show some properties of representations of linearly reductive group schemes in analogy to what is known over fields. Throughout this section, we assume that~$G$ is a flat affine group scheme over a commutative ring~$k$ unless specified otherwise. Under this assumption, the category of~$G$-modules is an abelian category (\cite[Proposition 2]{seshadri}). 

The flatness assumption on~$G$ guarantees that the fixed point functor~$(\cdot)^G$ is left exact so we may perform cohomology: similar to group cohomology of abstract groups, we may also obtain cohomology functors derived from the fixed point functor because the category of~$G$-modules is abelian and contains enough injective objects. The cohomology groups can be computed using Hochschild's complex: 
\begin{equation*}
    M \overset{\partial_0}{\longrightarrow} M \otimes_k k[G] \overset{\partial_1}{\longrightarrow} M\otimes_k k[G] \otimes_k k[G] \overset{\partial_2}{\longrightarrow} M\otimes_k k[G] \otimes_k k[G] \otimes_k k[G] \longrightarrow \cdots.
\end{equation*}\noindent Here~$\partial_0 \coloneqq \Delta_M - \identity_M$. See Chapter 4 of \cite{Jantzen} and also \cite[Section 2]{Margaux2009} for details. Let~$Z^i(G, M) \coloneqq \ker(\partial_i)$ denote the \emph{i-cocycles}, and~$B^i(G, M) = \im(\partial_{i-1})$ denote the \emph{i-coboundaries}. Define the \emph{i-th Hochschild cohomology group}~$H^i(G, M) \coloneqq Z^i(G, M)/B^i(G,M)$. Then we have the following exact sequence: 
\begin{equation}\label{hochschild_exact}
    0 \longrightarrow M^G = H^0(G, M) \longrightarrow M \overset{\partial_0}{\longrightarrow} Z^1(G, M) \longrightarrow H^1(G, M) \longrightarrow 0.
\end{equation}

We start with a cohomological description of linearly reductive groups.
\begin{defn}\label{def:LRG}
    A group~$G$ is linearly reductive if the first Hochschild cohomology group $H^1(G, V)$ vanishes for all~$G$-modules~$V$.
\end{defn}

We will see in the next section that the definition of a reductive group~$G$ depends on the nature of~$G$ over its \emph{geometric fibres}, i.e., the base changes~$G \times_k \overline{\kappa(x)}$, where~$\overline{\kappa(x)}$ denotes the algebraic closure of the residue field at a point~$x \in \spec(k)$. We will also explore how linear reductivity is related to linearly reductivity over geometric fibres.

\begin{prop}\label{LRG_base_change}
    Let~$G$ be a linearly reductive group over a commutative ring~$k$. Let~$A$ be a~$k$-algebra. 
    \begin{enumerate}[label = ($\alph*$)]
        \item The base change~$G_A = G \times_k A$ is a linearly reductive group over~$A$. In particular, if~$G$ is linearly reductive, then~$G$ is linearly reductive over all geometric fibres.
        \item Let~$A$ be faithfully flat over~$k$. If the base change~$G_A = G \times_k A$ is linearly reductive over~$A$, then~$G$ is linearly reductive over~$k.$
    \end{enumerate}
\end{prop}

\begin{proof}
    ($a$). Let~$V$ be a~$G_A$-module. By restriction of scalars, we may also view~$V$ as a~$G_k$-module, denoted by~$V_k$. Consider the comodule map~$\Delta_V \colon V \rightarrow V \otimes_A A[G].$ By definition,~$V^{G_A} = \{v \in V\mid \Delta_V(v) = v \otimes 1\}.$ Moreover, note that~$V \otimes_A A[G] \cong V \otimes_A A \otimes_k k[G] \cong V_k \otimes_k k[G].$ Here the second isomorphism makes sense since~$V \otimes_A A \cong V$ as both~$k$-modules and~$A$-modules. Then similarly,~$V_k^{G_k} = \{v \in V\mid \Delta_V(v) = v \otimes 1\} = V^{G_A}.$

    Now consider a surjective homomorphism of~$G_A$-modules~$\pi: M \twoheadrightarrow N$. Note that we may view~$M$ and~$N$ as~$G_k$-modules, denoted by~$M_k$ and~$N_k$. By assumption,~$\pi \colon M_k \twoheadrightarrow N_k$ induces~$\pi\mid_{M_k^{G_k}}\colon M_k^{G_k} \twoheadrightarrow N_k^{G_k}.$ Thus,~$\pi$ induces~$\pi\mid_{M^G} \colon M^G \twoheadrightarrow N^G.$

    ($b$). Let~$V \twoheadrightarrow W$ be a surjective homomorphism of~$G$-modules. By assumption, the surjective homomorphism of~$G_A$-modules~$V \otimes_k A \twoheadrightarrow W \otimes_k A$ induces~$(V \otimes_k A)^{G_A} \twoheadrightarrow (W \otimes_k A)^{G_A}$. Moreover, since~$A$ is faithfully flat over~$k$,~$(\cdot)^G$ is also left exact. Since~$(\cdot)^G$ commutes with flat base changes (\cite[Lemma 2]{seshadri}), we obtain a surjective homomorphism of~$G$-modules~$V^G \twoheadrightarrow W^G$. Therefore, ~$G$ is linearly reductive over~$k$.
\end{proof}

This shows that there are examples of groups that are linearly reductive over fields of characteristic zero, but fail to be linearly reductive over commutative rings. For example,~$\operatorname{GL}_n$ is linearly reductive over fields of characteristic zero but not linearly reductive over~$\mathbb{Z}$. 

\subsection{Reynolds operator}

Over fields, one interesting property of linearly reductive groups is that a Reynolds operator exists. We also explore this property over rings.

\begin{defn}\label{def_reynolds}
    Let~$V$ be a~$G$-module. A \emph{Reynolds operator}~$\mathcal{R}_{G}: V \rightarrow V^G$ is a~$G$-equivariant homomorphism such that~$\mathcal{R}_{G}(V) = V^G$ and~$\mathcal{R}_{G}\mid_{V^G} = \operatorname{id}.$ That is, if a Reynolds operator exists, there exists a~$G$-stable complement of~$V^G$ in~$V$. A \emph{functorial Reynolds Operator} is a choice of Reynolds operator for all~$G$-modules such that the following diagram is commutative.
    \begin{equation*}\label{functorial_Reynolds}
        \begin{tikzcd}
            & V \arrow{r}{\mathcal{R}_G}\arrow{d}{\varphi} & V^G \arrow{d}{\varphi\mid_{V^G}} \\
            & W \arrow{r}{\mathcal{R}_G} & W^G
        \end{tikzcd}
    \end{equation*}
\end{defn}

First, we prove the functorial property of the Reynolds operator.

\begin{prop}\label{reynolds_to_LR}
    Consider the~$G$-module~$k[G]$, on which~$G$ acts by left translation. If the trivial~$G$-module~$k$ has a~$G$-stable complement~$P$ in~$k[G]$, then~$H^1(G, V) = 0$ for all~$G$-modules. That is,~$G$ is linearly reductive.
\end{prop}

\begin{proof}
    For any~$G$-module~$V$,
    \begin{equation*}
        H^1(G, V\otimes k[G]) = H^1(G, V \otimes_k (k \oplus P)) = H^1(G, V) \oplus H^1(G, V\otimes_k P) = 0
    \end{equation*} by \cite[I.4.7, 4.15]{Jantzen}. Therefore,~$H^1(G, V) = 0$ for all~$G$-modules~$V$.
\end{proof}

If the trivial~$G$-module~$k$ has a~$G$-stable complement~$P$ in~$k[G]$, we can define an~$G$-equivariant homomorphism~$I \colon k[G] \to k$ with respect to left translation by letting~$I(a) = a$ for~$a \in k$ and~$I(p) = 0$ for~$p \in P$. Since the left and right translation commmute, we can show by the following lemma that~$I$ is~$G$-equivariant.

\begin{lemma}\label{unique_complement}
    If the trivial~$G$-module~$k$ has a~$G$-stable complement~$P$ in~$k[G]$, then~$P$ must be unique.
\end{lemma}

\begin{proof}
     Suppose by contradiction,~$Q$ is another complement. Then we get a nonzero map~$\phi\coloneqq I\mid_Q \colon Q \rightarrow k$. Let~$v$ be a nonzero element in the image of~$\phi$ and define~$\psi \colon k \rightarrow k$ by~$\psi(a) = av.$ Define~$M$ by the pull-back square
        \begin{tikzcd}
            &M \arrow[r, "\pi", twoheadrightarrow] \arrow{d}{}&k\arrow{d}{\psi}\\
            &Q \arrow{r}{\phi} &k
        \end{tikzcd}. Since~$H^1(G, \ker(\pi)) = 0$, any nonzero~$a \in k$ lifts to a nonzero invariant in~$M$ and therefore, we may find a nonzero invariant in~$Q$. However,~$Q^G = 0$. This leads to a contradiction.
\end{proof}

\begin{prop}\label{lem_functorial}
    Let~$G$ be a flat affine group scheme over a commutative ring~$k$. If a Reynolds operator exists for~$k[G]$, then there exists a functorial Reynolds operator for~$G$.
\end{prop}

\begin{proof} This statement is well-known if~$k$ is a field, e.g., see \cite{Gulbrandsen}, and the proof generalizes to a commutative ring~$k$.

    If a Reynolds operator exists for the~$G$-module~$k[G]$, then there exists a~$G$-equivariant homomorphism~$I \colon k[G] \rightarrow k[G]^G = k$ such that~$I(1) = 1$ with respect to both left and right translation of~$G$ by Lemma \ref{unique_complement} and the discussion above it. For a~$G$-module~$V$, define~$\mathcal{R}_G$ as the following composition of maps:
    \begin{equation*}
        \mathcal{R}_G\colon V \overset{\Delta_V}{\longrightarrow} V \otimes_k k[G] \overset{\operatorname{id}\otimes I}{\longrightarrow} V \otimes_k k \cong V.
    \end{equation*} 
    
    First note that for any~$v \in V^G$,~$\mathcal{R}_G(v) = ((\operatorname{id}\otimes I)\circ \Delta_V)(v) = (\operatorname{id}\otimes I)(v \otimes 1) = vI(1) = v$. To check that~$\mathcal{R}_G$ is a Reynolds operator, we then need to check that it is~$G$-equivariant. Let~$s$ denote a~$G$-equivariant section of~$I$. Consider the following diagram:
    \begin{center}
        \begin{tikzcd}
            &V \arrow{r}{\Delta_V} \arrow{d}{\Delta_V} & V\otimes_k k[G] \arrow{r}{\operatorname{id}\otimes I} \arrow{d}{\operatorname{id}\otimes \Delta_G} &V \otimes k \arrow{d}{\operatorname{id}\otimes s}
        \\
        & V \otimes_k k[G] \arrow{r}{\Delta_V \otimes \operatorname{id}} &V \otimes_k k[G] \otimes_k k[G] \arrow{r}{\operatorname{id}\otimes I \otimes \operatorname{id}} & V \otimes_k k[G]
    \end{tikzcd}
    \end{center}
    \noindent Note that the left square is commutative by the associativity of the group law and the right square is commutative because~$I$ is~$G$-equivariant with respect to the right action. Therefore, the outer square is commutative, whence,~$\mathcal{R}_G$ is~$G$-equivariant. 
    
    Then we check that the image of~$V$ under~$\mathcal{R}_G$ is~$V^G$. That is, we show that~$\Delta_V(\mathcal{R}_G(v)) = \mathcal{R}_G(v) \otimes 1.$ Consider the following diagram:
    \begin{center}
        \begin{tikzcd}
            &V \arrow{r}{\Delta_V} \arrow{d}{\Delta_V} & V\otimes_k k[G] \arrow{r}{\operatorname{id}\otimes I} \arrow{d}{\operatorname{id}\otimes \Delta_G} &V \otimes k\arrow{d}{\operatorname{id}\otimes s}
        \\
        & V \otimes_k k[G] \arrow{r}{\Delta_V \otimes \operatorname{id}} \arrow{dr}{\operatorname{id}\otimes I}&V \otimes_k k[G] \otimes_k k[G] \arrow{r}{\operatorname{id} \otimes \operatorname{id}\otimes I} & V \otimes_k k[G]\\
        & &V\otimes k \arrow{ur}{\Delta_V} &
    \end{tikzcd}
    \end{center}
    The left square is commutative as argued before and the right square is commutative since~$I$ is~$G$-equivariant with respect to the left action of~$G$. Moreover, the lower triangle is also commutative. Indeed,~$(\operatorname{id}  \otimes \operatorname{id}\otimes I) \circ (\Delta_V \otimes \operatorname{id}) = ((\operatorname{id} \otimes \operatorname{id})\circ \Delta_V)\otimes (I \circ \operatorname{id}) =(\Delta_V\circ \operatorname{id}) \otimes (\operatorname{id} \circ I) = \Delta_V\circ (\operatorname{id}\otimes I).$ Therefore, the outer pentagon is commutative. Hence,
    \begin{align*}
        \Delta_V (((\operatorname{id}\otimes I)\circ \Delta_V)(v)) &= \Delta_V(\mathcal{R}_G(v))\\
        &= ((\operatorname{id}\otimes s)\circ(\operatorname{id}\otimes I) \circ \Delta_V) (v)\\
        &= (\operatorname{id} \otimes s)(\mathcal{R}_G(v)) \\
        &= \mathcal{R}_G(v) \otimes 1.
    \end{align*}

    It remains to check that~$\mathcal{R}_G$ is functorial. Let~$\varphi \colon V \rightarrow W$ be a homomorphism of~$G$-modules. We see that the outer square of the following diagram commutes: the left square commutes since~$\varphi$ is~$G$-equivariant and the right square commutes for trivial reasons.
    \begin{center}
        \begin{tikzcd}
            &V \arrow{r}{\Delta_V} \arrow{d}{\varphi}&V \otimes_k k[G] \arrow{r}{\operatorname{id}\otimes I} \arrow{d}{\varphi \otimes \operatorname{id}} &V \otimes k \arrow{d}{\varphi \otimes \operatorname{id}}\\
            & W \arrow{r}{\Delta_W} & W \otimes_k k[G] \arrow{r}{\operatorname{id}\otimes I} &W \otimes k
        \end{tikzcd}
    \end{center} Therefore,~$\mathcal{R}_G$ is functorial.
\end{proof}

Now, in order to study whether a functorial Reynolds operator exists for a linearly reductive group, we first come back to some properties of Hochschild cohomology. 
\begin{lemma}\label{cohomology_lemma}
    Let~$k$ be a commutative ring and let~$M$ be a~$G_k$-module. Let~$M = \varinjlim_i M_i$ be a direct limit of~$G_k$-modules~$M_i$, then~$$\varinjlim_i H^j(G, M_i) \overset{\sim}{\longrightarrow} H^j(G, M) \textnormal{ for all }j \geq 0$$ and~$$\varinjlim_i Z^j(G, M_i) \overset{\sim}{\longrightarrow} Z^j(G, M) \textnormal{ for all }j \geq 1.$$
\end{lemma}
\begin{proof}
    For the first claim, see \cite[Lemma 4.17]{Jantzen}. For the second claim, note that the inductive limit is left exact and therefore preserves kernel. See \cite[Proposition 5.25, 5.33]{Rotman}.
\end{proof}

\begin{theorem}\label{theorem:reynolds} Let~$G$ be a smooth affine algebraic group scheme over a commutative ring~$k$. 
    \begin{enumerate}[label = ($\alph*$)]
        \item  If ~$G$  has a functorial Reynolds operator, then it is linearly reductive.
        \item Assume that~$k$ is a Dedekind domain. Then a linearly reductive group~$G$ over~$k$ has a functorial Reynolds operator.
    \end{enumerate}
\end{theorem}

\begin{proof}
   ~$(a)$ This is immediate from Proposition \ref{reynolds_to_LR}.

   ~$(b)$ 
    By Proposition \ref{lem_functorial}, it is enough to show that~$k = k[G]^G \subseteq k[G]$ has a~$G$-stable complement. To show this, it suffices to prove that for any~$G$-submodule~$V$ of~$k[G]$ that is finitely generated over~$k$, its invariant submodule~$V^G$ has a~$G$-stable complement. Indeed, since~$k$ is Noetherian, ~$k[G]$ is locally finite by Lemma~\ref{locally_finiteness}. Write~$k[G] = \varinjlim_i V_i$, where~$V_i$ is a~$G$-module that is finitely generated over~$k$. Suppose that~$V_i = V_i^G \oplus V_i^c$, where~$V_i^c$ is~$G$-stable. By exact sequence~\ref{hochschild_exact} and the fact that~$H^1(G, V_i) = 0$, it follows that~$\partial_{0, i} \mid_{V_i^c} \colon V_i^c \to Z^1(G, V_i)$ is an isomorphism of~$G$-modules, where~$\partial_{0,i} \coloneqq \Delta_{V_i} - \textnormal{id}_{V_i}$ is the 0-th differential map from the Hochschild's complex. Since~$\partial_{0,i}$ is functorial, so is its inverse, denoted by~$s_i$. Thus, the collection of~$G$-modules~$\{s_i(Z^1(G,V_i))\}_i$ forms a direct system and~$s \colon \varinjlim_i Z^1(G, V_i) \overset{\sim}{\to} \varinjlim_i s_i(Z^1(G,V_i))$ is the inverse to~$\partial_0$.
    Then 
    \begin{align*}
        k[G] &= \varinjlim_i V_i = \varinjlim_i (H^0(G, V_i) \oplus s_i(Z^1(G, V_i)))\\ &= (\varinjlim_i H^0(G, V_i)) \oplus (\varinjlim_i s_i(Z^1(G, V_i))) \\
        &\cong H^0(G, \varinjlim_i V_i) \oplus s(\varinjlim_i Z^1(G, V_i))\\ &\cong H^0(G, k[G]) \oplus Z^1(G, k[G]) \\&= k \oplus Z^1(G, k[G]),
    \end{align*} where the isomorphisms follow from Lemma~\ref{cohomology_lemma} and the arguments above.
    
    Since~$G$ is smooth and has connected geometric fibres, by \cite[Lemma 1]{seshadri}, the~$G$-module~$k[G]$ is projective as a~$k$-module. Since~$k$ is a Dedekind domain, by definition, any finitely generated submodule of~$k[G]$ is also projective. Then it suffices to prove the statement for any~$G$-module~$V$ that is finitely generated and projective over~$k$. Consider the exact sequence of~$G$-modules from ~\ref{hochschild_exact}:
    \begin{center}
        \begin{tikzcd}
            &0 \arrow{r}{} &V^G \arrow{r}{} &V \arrow{r}{} &Z^1(G,V) \arrow{r}{} &0
        \end{tikzcd}
    \end{center} 
    Applying the~$\Hom_k(\cdot, V^G)$ functor yields
    \begin{equation*}
        0 \to \Hom_k(Z^1(G, V), V^G) \to \Hom_k(V, V^G) \to \Hom_k(V^G, V^G) \to \Ext^1_k(Z^1(G,V), V^G) \to \cdots
    \end{equation*} Note that this is indeed an exact sequence of~$G$-modules:
    since~$k$ is a Dedekind domain and a submodule of a finitely generated module is finitely generated, both~$V^G \subseteq V$ and~$Z^1(G, V) \subseteq V \otimes_k k[G]$ are finitely generated and projective, so each term in the above exact sequence carries a~$G$-module structure by Remark \ref{alg_operation};
    each map is~$G$-equivariant since~$\Hom(\cdot, V^G)$ is functorial.
    By \cite[Exercise 2.5.2]{wiebel},~$\Ext^1_k(Z^1(G,V), V^G)$ vanishes since $Z^1(G,V)$ is projective. Therefore, we have the short exact sequence of~$G$-modules:
    \begin{equation*}
        0 \to \Hom_k(Z^1(G, V), V^G) \to \Hom_k(V, V^G) \to \Hom_k(V^G, V^G) \to 0
    \end{equation*} Now, since~$G$ is linearly reductive, applying~$(\cdot)^G$ to the short exact sequence gives 
    \begin{equation*}
        0 \to \Hom_k(Z^1(G, V), V^G)^G \to \Hom_k(V, V^G)^G \overset{\textnormal{res}}{\to} \Hom_k(V^G, V^G)^G \to 0
    \end{equation*} Thus, the map~$\res \colon \Hom_k(V, V^G)^G \to \Hom_k(V^G, V^G)^G \colon \varphi \mapsto \varphi\mid_{V^G}$ is surjective. Hence,~$\identity \in \Hom_k(V^G, V^G)^G$ has a preimage~$\phi \in \Hom_k(V, V^G)$. That is,~$\phi \colon V \to V^G$ is a~$G$-equivariant homomorphism and~$\phi\mid_{V^G} = \identity$, whence~$V^G$ has a~$G$-stable complement~$\ker(\phi)$ in~$V$.
\end{proof}

\begin{rem}
    Over fields, a~$G$-module~$V$ of a linearly reductive group~$G$ is \emph{completely reducible}. That is, any~$G$-submodule~$W$ of~$V$ has a~$G$-stable complement in~$V$. Over commutative rings, this is in general not true. In fact, over~$\mathbb{Z}$, there are no groups that satisfy this property. Indeed, let a group~$G$ act trivially on~$\mathbb{Z}$. Then~$2\mathbb{Z}$ is a~$G$-submodule of~$\mathbb{Z}$ but does not have a~$G$-stable complement in~$\mathbb{Z}$.
\end{rem}

\subsection{The Cohen--Macaulay property of rings of invariants}\label{subsection:cm}

Recall that a module~$M$ over a Noetherian local ring~$k$ is \emph{Cohen--Macaulay} if~$\operatorname{depth}(M) = \dim(M)$, where~$\operatorname{dim}(M)$ is the Krull dimension of~$k/\operatorname{Ann}(M)$. If~$k$ is not necessarily local, then~$M$ is Cohen--Macaulay if for all maximal ideals~$\mathfrak{m}$ such that~$M_{\mathfrak{m}}\neq 0$,~$M_{\mathfrak{m}}$ is Cohen--Macaulay as a~$k_{\mathfrak{m}}$-module. Moreover, a ring is \emph{Cohen--Macaulay} if it is Cohen--Macaulay as a module over itself.

The Cohen--Macaulay property measures the local equidimensionality of rings and therefore has various applications in algebraic geometry.  Over fields, Hochster and Roberts showed that for a linearly reductive group~$G$ over~$k$ and a rational representation~$V$, the ring of invariants~$k[V]^G$ is Cohen--Macaulay (\cite{Hochster-Robert}). Later, this was also shown for the coordinate rings of smooth affine~$G$-varieties by Hochster and Huneke (\cite[Proposition 4.12]{hochster-huneke}). When~$k = \mathbb{Z}$ and~$G$ is finite, it was studied in \cite[Corollay 6.2.12]{AA2018} that~$\mathbb{Z}[x_1, \dots, x_n]^G$ is Cohen--Macaulay if and only if~$\mathbb{F}_p[x_1, \dots, x_n]^G$ is Cohen--Macaulay for all primes~$p$ dividing the order of~$G$. We will generalize Hochster and Huneke's result to Dedekind domains.

The following lemma from \cite{BOUCHIBA} is useful for the rest of the section.
\begin{lemma}\label{cm_flat}
    Let~$R$ be a Cohen--Macaulay ring. Let~$A$ and~$B$ be~$R$-algebras. Assume that~$B$ is flat over~$R$. Then the following assertions are equivalent:
    \begin{enumerate}[label=(\alph*)]
        \item~$A \otimes_R B$ is a Cohen--Macaulay ring.
        \item~$A_{\mathfrak{p}}$ and~$B_{\mathfrak{q}}$ are Cohen--Macaulay rings for any prime ideals~$\mathfrak{p}$ of~$A$ and~$\mathfrak{q}$ of~$B$ such that~$\mathfrak{p} \cap R = \mathfrak{q} \cap R$.
    \end{enumerate} 
\end{lemma}

\begin{proof}
    See Corollary 2.10 of \cite{BOUCHIBA}.
\end{proof}

\begin{theorem}\label{CM}
    Let~$G$ be a linearly reductive group over a Dedekind domain~$k$. Let~$X$ be a smooth affine algebraic~$G$-scheme over~$k$ with connected generic fibre. Let~$k[X]$ be the coordinate ring of~$X$. Then~$k[X]^G$ is Cohen--Macaulay.
\end{theorem}

\begin{proof}
    By the assumption on~$X$,~$k[X] \subseteq K[X_K]$ is an integral domain and therefore,~$k[X]^G \subseteq k[X]$ is also an integral domain.
    We will show that~$k[X]^G_{\mathfrak{M}}$ is Cohen--Macaulay for all maximal ideals~$\mathfrak{M}$ of~$k[X]^G$.

    We first prove the case that~$\mathfrak{M} \cap k = (0)$. By \cite[Prposition 4.12]{hochster-huneke},~$k[X]^G \otimes_k K  \cong K[X]^{G_K}$ is Cohen--Macaulay. Then in Lemma \ref{cm_flat}, by letting~$A = k[X]^G$,~$B = K$, the fact that~$k[X]^G \otimes_k k_{(0)} \cong k[X]^G \otimes_k K$ is Cohen--Macaulay implies that~$k[X]^G_{\mathfrak{M}}$ is Cohen--Macaulay for any maximal ideal~$\mathfrak{M}$ such that~$\mathfrak{M} \cap k = (0) \cap k = (0)$, since the only prime ideal of~$K$ is~$(0)$. Hence,~$k[X]^G_{\mathfrak{M}}$ is Cohen--Macaulay for all maximal ideals~$\mathfrak{M}$ such that~$\mathfrak{M} \cap k = (0).$
    
     If~$\mathfrak{M} \cap k \neq (0)$, then~$\mathfrak{m} \colonequals \mathfrak{M} \cap k$ is a maximal ideal of~$k$, since~$\mathfrak{M} \cap k$ is a nonzero prime ideal of~$k$, which has Krull dimension one. By \cite[Lemma 2]{seshadri} , for any maximal ideal~$\mathfrak{m}$ of~$k$,~$k[X]^G \otimes_k k_{\mathfrak{m}} \cong k_{\mathfrak{m}}[X \times_{\operatorname{Spec}(k)} \operatorname{Spec}(k_{\mathfrak{m}})]^{G_{k_{\mathfrak{m}}}}$. Therefore, it suffices to prove that $S \coloneqq k[X]^G$ is Cohen-Macaulay when $k$ is a principal ideal domain: in Lemma \ref{cm_flat}, by letting~$A = k[X]^G$ and~$B = k_{\mathfrak{m}}$, the fact that~$k[X]^G \otimes_k k_{\mathfrak{m}}$ is Cohen--Macaulay implies that~$k[X]^G_{\mathfrak{M}}$ is Cohen--Macaulay for all maximal ideals~$\mathfrak{M}$ such that~$\mathfrak{M} \cap k = \mathfrak{m}$.

     Let~$\mathfrak{N}$ be a maximal ideal of~$S$. If~$\mathfrak{N} \cap k \neq (0)$, then~$\mathfrak{n} \colonequals \mathfrak{N} \cap k$ is a maximal ideal of ~$k.$ Indeed,~$\mathfrak{N} \cap k$ is a nonzero prime ideal of~$k$, which has Krull dimension one.  Let~$\mathfrak{n}$ be generated by~$x$ for some~$x \in k$. By \cite[Proposition 4.18]{Jantzen}, 
    \begin{equation*}
        S/xS = k[X]^G \otimes k/\mathfrak{n} \cong ((k/\mathfrak{n})[X])^{G_{k/\mathfrak{n}}}
    \end{equation*} which is Cohen--Macaulay by \cite[Proposition 4.12]{hochster-huneke}. Then
    \begin{equation*}
        S_{\mathfrak{N}}/xS_{\mathfrak{N}} \cong (S/xS)_{\mathfrak{N}/(xS\cap \mathfrak{N})},
    \end{equation*} is also Cohen--Macaulay by definition. Since~$x$ is not a zero-divisor of~$S$, it is also not a zero-divisor of~$S_{\mathfrak{N}}$. Thus,~$x$ forms a regular sequence of~$S_{\mathfrak{N}}$. Hence,~$S_{\mathfrak{N}}$ is Cohen--Macaulay by Theorem~\cite[Theorem 2.1.3(a)]{bruns_herzog}. 
    
    It remains to prove the claim for the case that~$\mathfrak{N} \cap k = (0)$. We can prove the same way as in the second paragraph.
\end{proof}

By taking~$X = \mathbb{A}^n_k$, we obtain the immediate corollary:
\begin{cor}\label{cm_poly}
    Let~$k[x_1, x_2, \cdots, x_n]$ be a polynomial ring over a Dedekind domain~$k$ and let ~$G$ be a linearly reductive group over~$k$. The ring of invariants~$k[x_1, x_2, \cdots, x_n]^G$ is Cohen--Macaulay. 
\end{cor}

\begin{rem}
    Note that Theorem \ref{CM} can also be obtained by the fact that a pure subring of a regular ring is Cohen--Macaulay (see \cite[Theorem 4.1.12]{HH20}, \cite[Corollary 4.3]{HM18}) after we proved Theorem \ref{theorem:reynolds}(b).
\end{rem}

\section{Different Notions of Reductive Groups}\label{sec_rg}
In this section, we first recall some notions of reductive group schemes that are known in the literature and study the relationships between them. Then we compare them to our new notion of linearly reductive groups. Throughout this section, we assume that~$G$ is a flat affine algebraic group scheme over a commutative ring~$k,$ unless specified otherwise. 

\subsection{Reductive group schemes}\label{sec_reductive}

Recall that over a field~$K$, a linear algebraic group~$G$ is \emph{reductive} if its unipotent radical, the largest connected unipotent normal subgroup of~$G$ with respect to inclusion, is trivial. Over an arbitrary ring, we adopt the notion of reductivity introduced in \cite{SGA73}. An affine group scheme~$G$ over~$k$ is \emph{reductive} if
    \begin{enumerate}[label =$\left(\alph*\right)$]
        \item $G$ is smooth over~$k$.
        \item For all~$x\in \operatorname{Spec}(k)$, the geometric fibre~$G \times_k \overline{\kappa(x)}$ is connected and reductive, where~$\overline{\kappa(x)}$ denotes the algebraic closure of the residue field.
    \end{enumerate}

The notion of geometric reductivity (over fields) originated in Mumford's Conjecture: 
\begin{conj*}(\cite[Mumford's Conjecture]{git})
    Let~$G$ be a semisimple algebraic group and~$V$ a~$G$-module. If~$V_0$ is a~$G$-invariant submodule of codimension one, then for some ~$d\in \mathbb{Z}_{+}$,~$V_0\cdot S^{d-1}(V) \subseteq S^d(V)$ has a~$G$-stable invariant complement of dimension one.
\end{conj*}
\noindent Haboush gave an affirmative answer to the Mumford's Conjecture by showing that reductive groups are geometrically reductive. There the geometric reductivity is formulated with the dual of the representations so that one can work with polynomials. Seshadri extended this definition to Noetherian rings in \cite{seshadri}: 
\begin{defn}\label{seshadri_gr}
    A smooth affine algebraic group scheme~$G$ over a Noetherian ring~$k$ with connected geometric fibres is \emph{geometrically reductive} if for any finitely generated free~$G$-module~$V$ and  any field~$F$ that is also a~$k$-algebra on which~$G$ acts trivially, for any nonzero~$v\in (V \otimes_k F)^{G_F}$, there exist a positive integer~$d$ and~$f \in S^d(V^{\ast})^G$ such that~$f(v) \neq 0$.
\end{defn}
\noindent  It was also shown that reductive groups are geometrically reductive over Noetherian rings \cite[Theorem 1]{seshadri}. 
We will show in the next section that over a regular ring of Krull dimension at most two, geometrically reductive groups are reductive.

Although over fields Haboush's formulation of geometric reductivity is parallel with Mumford's conjecture, over rings Seshadri's formulation is not. 
There are some disadvantages to working with Seshadri's definition over an arbitrary base. 
For example, finiteness and freeness of ~$G$-modules are not always preserved via restriction of scalars. 
This provides a restrictive view of representations. 
Moreover, it does not give enough information for~$G$-modules over the base ring. In \cite{Franjou_VanderKallen}, Franjou--van der Kallen formulated power reductivity over arbitrary bases which is compatible with the setup in Mumford's Conjecture: 
\begin{defn}\label{PR_1}
    A flat affine algebraic group scheme~$G$ over a commutative ring~$k$ is \emph{power reductive} if for any surjective~$G$-module homomorphism~$\varphi: M \twoheadrightarrow k$, there exists a positive integer~$d$ such that the map induced by the~$d$-th symmetric power
    \begin{equation*}
        S^d(\varphi): S^d(M) \longtwoheadrightarrow S^d(k)
    \end{equation*} splits. That is,~$\ker(S^d(\varphi))$ has a~$G$-stable complement in~$S^d(M)$.
\end{defn}

Over fields, power reductivity is equivalent to geometric reductivity (\cite{Kallen2004ARG}). Over a commutative unital ring, a reductive group is power reductive (\cite[Theorem 12]{Franjou_VanderKallen}) and that a power reductive group is geometrically reductive (\cite[Lemma 12]{PR_vdK}). Over a discrete valuation ring, a geometrically reductive group is power reductive (\cite[Lemma 13]{PR_vdK}).

\subsection{Geometrically reductive groups over regular rings of dimension at most two} 
Throughout this section, we assume~$G$ is a smooth affine algebraic group scheme over a Noetherian ring~$k$ with connected geometric fibres. Under such assumptions, the local finiteness property is satisfied for all~$G$-modules \cite[Lemma 1]{seshadri}. Consider the following generalized notion of geometric reductivity using projective modules:

\begin{defn}\label{GR_Proj}
    A group~$G$ is \emph{strongly geometrically reductive} over~$k$ if for any finitely generated projective~$G$-module~$M$, any~$k$-algebra~$F$ that is a field, and any nonzero~$v \in (M \otimes_k F)^{G_F}$, there exists a positive integer~$d$ and~$f \in S^d(M^{\ast})^{G}$, such that~$f(v) \neq 0$. 
\end{defn}

Notice that this new version of geometric reductivity coincides with Seshadri's definition of geometric reductivity over fields, but is in general stronger than Seshadri's definition. It is easier to work with this stronger definition in order to apply the~$G$-equivariant resolution property. It turns out that if the base is a regular ring of Krull dimension at most two, the notions of geometric reductivity and strong geometric reductivity coincide.

\begin{lemma}\label{strong_GR}
    Over a regular ring~$k$ of Krull dimension at most two, a group~$G$ is geometrically reductive if and only if it is strongly geometrically reductive.
\end{lemma}

\begin{proof}
    The forward direction is obvious. We will show the converse here.
    Let~$V$ be a ~$G$-module that is finitely generated and projective over~$k$. 
    Then~$V$ is a direct summand of some finitely generated free module~$W$, i.e.,~$W:= V \oplus I$ for some~$k$-module~$I$. 
    We endow~$W$ with a~$G$-module structure that is compatible with that on~$V$ by letting~$G$ act on ~$I$ trivially and act on~$V$ as before.

    Let~$F$ be a field that is also a~$k$-algebra on which~$G$ acts trivially. By Definition~\ref{seshadri_gr}, for all~$v \in (V \otimes_k F)^{G_F} \subseteq (W \otimes_k F)^{G_F}$, there exists a positive integer~$d$ and~$f \in S^d(W^{\ast})^{G}$, such that~$f(v)\neq 0$. 
    Moreover, by Lemma~\ref{dual_proj}, the inclusion map of~$G$-modules~$i: V \hookrightarrow W$ induces a~$G$-module homomorphism~$\pi: W^{\ast} \rightarrow V^{\ast}$. Applying~$S^d(\cdot)$ further induces~$S^d(\pi): S^d(W^{\ast}) \rightarrow S^d(V^{\ast}).$ Since~$S^d(\pi)((S^d(W^{\ast}))^{G}) \subseteq S^d(V^{\ast})^G$, there exists~$h\colonequals S^d(\pi)(f) \in S^d(V^{\ast})^G$ such that~$h(v)\neq 0$.
\end{proof}

To compare the notion geometric reductivity to the other notions of reductivity, we may compare the notion of strongly geometric reductivity to the rest. 

\begin{lemma}~\label{res_dual_vsp}
    Let~$V$ be a free module over a integral domain~$D$ and let~$F$ be a field that contains~$D$. Let~$f \in (V \otimes_D F)^{\ast}$ be nonzero. Then~$f\vert_{V}$ is nonzero.
\end{lemma}
\begin{proof}
    Let~$\mathcal{B}$ be a basis for~$V$ over~$D$. Then~$\mathcal{B}' = \{v \otimes 1 \vert v \in \mathcal{B}\}$ is an~$F$-basis for~$V \otimes_D F.$ If~$f \in (V \otimes_D F)^{\ast}$ is nonzero, then there exists~$v \otimes 1 \in \mathcal{B}'$, s.t.,~$f(v \otimes 1) \neq 0$. Therefore,~$f\vert_{V}(v) \neq 0$ for this~$v \in \mathcal{B}$. Hence,~$f\vert_V \neq 0.$
\end{proof}

\begin{prop}\label{PR_GR}
    If~$k$ is a regular local ring of Krull dimension at most two, a strongly geometrically reductive group $G$ is power reductive.
\end{prop}

\begin{proof} The following proof closely mirrors the proof of \cite[Lemma 13]{PR_vdK} in which~$k$ is a discrete valuation ring.

    Let~$G$ be a strongly geometrically reductive group scheme over~$k$ with residue field~$F$. Let~$M$ be a~$G$-module over~$k$ and let~$\phi\colon M \twoheadrightarrow k$ be a surjective~$G$-module homomorphism. Choose~$m \in M$ such that~$\phi(m) = 1$. Since~$M$ is locally finite, we may assume without loss of generality that~$M$ is finitely generated. By the assumptions on~$G$ and~$k$,~$(G, \operatorname{Spec}(k))$ satisfies the~$G$-equivariant property (\cite[Lemma 2.5]{THOMASON1987}), so there exists a~$G$-module~$N$ that is finitely generated and projective over~$k$ such that~$N \overset{\psi}{\twoheadrightarrow} M$.
    
    Let~$v$ denote the following composition map of~$G$-modules:
    \begin{alignat*}{4}
        v \colon N  \overset{\psi}{\longrightarrow} &M & \overset{\phi}{\longrightarrow} k & \overset{q}{\longrightarrow} F \\
                  &  m                     & \longmapsto 1                     & \longmapsto 1
    \end{alignat*} Then we must have~$v\in \Hom_F(N \otimes_k F, F)^{G_F} = (N^{\ast} \otimes_k F)^{G_F}$. By applying the definition of strongly geometrically reductive, there exists~$f \in S^d((N^{\ast})^{\ast})^{G} \cong S^d(N)^{G}$, such that~$f(v) \neq 0$. Here the isomorphism follows from the fact that~$N$ is finitely generated and projective and therefore reflexive.

    Moreover, applying the symmetric~$d$-th power~$S^d(\cdot)$ gives
    \begin{equation*}
        S^d(v) \colon S^d(N) \overset{S^d(\psi)}{\longrightarrow} S^d(M) \overset{S^d(\phi)}{\longrightarrow} S^d(k) \overset{S^d(q)}{\longrightarrow} S^d(F)
    \end{equation*} Therefore,~$f(v) = S^d(v)(f) = S^d(q)\circ S^d(\phi\circ\psi)(f) \neq 0$. Since~$k$ is local,~$S^d(\phi \circ \psi)$ maps~$f$ to the unique (nonzero) lift of~$f$ under~$S^d(q)$. Moreover,~$S^d(\phi\circ\psi)(f) = u\cdot x^d$, where~$u$ is a unit of~$k$ and~$x^d$ is a generator of~$S^d(k)$. 
    Therefore, the set~$\{c\cdot f\mid c \in k\}$ is a~$G$-stable complement to~$\operatorname{ker}(S^d(\phi\circ\psi))$, so~$S^d(N) \rightarrow S^d(k)$ splits, whence~$S^d(M) \rightarrow S^d(k)$ also splits.
\end{proof}

To prove the converse, we first show that strong geometric reductivity is a local property.

\begin{theorem}\label{global_to_local}
    Let~$k$ be a regular ring of Krull dimension at most two. If~$G$ is strongly geometrically reductive over~$k$, then~$G$ is strongly geometrically reductive over~$k_{\mathfrak{p}}$, for all prime ideals~$\mathfrak{p}$.
\end{theorem}

\begin{proof} 
    Fix a prime ideal~$\mathfrak{p}$. Let~$M$ be a ~$G_{k_{\mathfrak p}}$-module that is finitely generated and projective over $k_{\mathfrak p}$, let~$F$ a~$k_{\mathfrak p}$-algebra that is a field, and let~$v_{F}\in (M \otimes_{k_{\mathfrak{p}}} F)^{G_F}$. To prove the theorem, we must show that there exist a positive integer~$d$ and an element~$h\in S^d_{k_{\mathfrak{p}}}(M^{\ast})^{G_{k_{\mathfrak{p}}}}$ such that~$h(v_F) \neq 0$. We break up the proof into three steps.

    \textbf{Step 1:} \textit{We show that there exists~$m^{\ast} \in M^{\ast}$ such that~$v_F$, viewed as an element of~$((M \otimes_{\loc{k}} F)^{\ast})^{\ast}$, maps~$m^{\ast}\otimes 1\in M^{\ast} \otimes_{\loc{k}} F$ to a nonzero element of~$F$.}

    By Lemma~\ref{dual_proj}, both the dual module~$M^{\ast}$ and the double dual module~$M^{\ast \ast}$ are finitely generated and free and have compatible~$G$-module structures. Moreover,~$M$ is reflexive, i.e.,~$M^{\ast\ast} \cong M$. This isomorphism is canonical, and so is~$G$-equivariant. Therefore, we have the isomorphisms~$M \otimes_{\loc{k}} F \cong M^{\ast \ast} \otimes_{\loc{k}} F\cong (M^{\ast} \otimes_{k_{\mathfrak{p}}} F)^{\ast}$,
    which are~$G$-equivariant by Lemma~\ref{dual_proj}. Thus we can view~$v_F\in M \otimes_{\loc{k}} F$ as a map~$v_F : M^{\ast} \otimes_{\loc{k}} F \rightarrow F$. Let~$D$ denote the image of~$k_{\mathfrak{p}}$ under the residue map~$k_{\mathfrak{p}} \to F$, so~$D$ is an integral domain. Then the base change of~$M^{\ast}$ to~$F$ factors through~$D$:
    \begin{equation*}
        \operatorname{id} \otimes_{k_{\mathfrak{p}}} F \colon M^{\ast} \overset{q}{\longtwoheadrightarrow}  M^{\ast} \otimes_{\loc{k}} D \overset{\iota}{\longhookrightarrow} M^{\ast} \otimes_{\loc{k}} F,
    \end{equation*} 
    where~$\iota \coloneqq \operatorname{id}\otimes_{D} F$ is injective because~$M^\ast$ is flat over~$\loc{k}$. Extension of scalars preserves~$G$-module structures on~$M^{\ast} \otimes_{\loc{k}} D$ and~$M^{\ast} \otimes_{\loc{k}} F$. By Lemma~\ref{res_dual_vsp} and the surjectivity of~$q$, there exists~$m^\ast\in M^\ast$ such that~$v_F(\iota \circ q(m^{\ast})) = v_F((\operatorname{id} \otimes_{k_{\mathfrak{p}}} F)(m^{\ast}))\neq 0$.

    \textbf{Step 2:} 
    \textit{We construct a ~$G_k$-submodule~$V$ of ~$(M^{\ast})_k$ containing~$m^{\ast}$ that is finitely generated over $k$. By the existence of equivariant resolution, we construct a ~$G_k$-module~$W$ that is finitely generated over $k$ such that~$W^{\ast} \twoheadrightarrow V \hookrightarrow (M^{\ast})_k$, and~$\pi(v_F) \in (W \otimes_k F)^{G_F}$ is not zero, where~$\pi$ is induced by the the map~$W^{\ast} \rightarrow (M^{\ast})_k$ as in Lemma~\ref{dual_proj}.}

    Recall that~$(M^{\ast})_k$ denotes~$M^{\ast}$ viewed as a~$k$-module via restriction of scalars, so~$(M^{\ast})_k$ inherits the~$G$-module structure.
    By the assumption on~$G$, any~$G_k$-module is locally finite (\cite[Lemma 1]{seshadri}), so there exists a finitely generated~$G_k$-submodule~$V \subseteq (M^{\ast})_k$ containing~$m^\ast$. Moreover, by Lemma \ref{lemma:thomason_equivariant}, there is a finitely generated and projective~$G$-module~$N$ and a~$G$-equivariant surjective homomorphism~$N \twoheadrightarrow V \subseteq (M^{\ast})_k$. Let~$W:= N^{\ast}$. By Lemma~\ref{dual_proj},~$W$ is finitely generated and projective over $k$ with a~$G$-module structure compatible with that on~$N$. Similarly, same assertion holds for~$W^{\ast} = N^{\ast\ast} \cong N$.

    The composition of maps~$W^{\ast} \overset{\sim}{\to} N \twoheadrightarrow V \hookrightarrow (M^{\ast})_k$ induces~$\alpha: W^{\ast}\otimes_k \loc{k} \rightarrow (M^{\ast})_k \otimes_k \loc{k}$. By Lemma \ref{dual_proj},~$W^{\ast} \otimes_k k_{\mathfrak{p}} \cong (W \otimes_k k_{\mathfrak{p}})^{\ast}$ and by Proposition \ref{base_change_representation},~$(M^{\ast})_k \otimes_k \loc{k}  \cong M^{\ast}$, and both of the isomorphisms preserve the~$G$-module structures.
    Now, consider the induced~$G$-equivariant map on the dual modules~$\alpha^{\ast}: M^{\ast\ast} \cong M \to (W \otimes_k k_{\mathfrak{p}})^{\ast\ast} \cong W \otimes_k k_{\mathfrak{p}}$,
    which further induces~$\pi:= \alpha^{\ast}\otimes_{k_{\mathfrak{p}}} F: M \otimes_{k_{\mathfrak{p}}} F \to W \otimes_k F$.

    Recall from Step 1 that~$v_F \in (M \otimes_{\loc{k}} F)^{G_F}$ and~$v_F((\operatorname{id} \otimes_{k_{\mathfrak{p}}} F)(m^{\ast})) \neq 0$ for some~$m^{\ast} \in M^{\ast}$. We now claim that~$\pi(v_F) \in (W \otimes_k F)^{G_F}$ and~$\pi(v_F)\neq 0$. The first part of claim is obvious since~$v_F \in (M \otimes_{\loc{k}} F)^{G_F}$. For second part of the claim, since~$W^{\ast} \twoheadrightarrow V \subseteq (M^{\ast})_k$, there exists a preimage~$w^{\ast} \in W^{\ast}$ of~$m^{\ast} \in V \subseteq (M^{\ast})_k$ as above, and then from the following commutative diagram
    \begin{center}
        \begin{tikzcd}
            &(W \otimes_{k} \loc{k})^{\ast} \arrow[r, rightarrow, "\alpha"]\arrow{d}{\operatorname{id} \otimes_{k_{\mathfrak{p}}} F} & M^{\ast}\arrow{d}{\operatorname{id} \otimes_{k_{\mathfrak{p}}} F}\\
            & (W \otimes_{k} \loc{k})^{\ast} \otimes_{\loc{k}} F \arrow{r}{\alpha \otimes_{\loc{k}} F} & M^{\ast} \otimes_{\loc{k}} F,
        \end{tikzcd}
    \end{center}
    we see that
    \begin{align*}
        \pi(v_F)((\operatorname{id} \otimes_{k_{\mathfrak{p}}} F)(w^{\ast})) &= v_F\circ (\alpha \otimes_{\loc{k}} F) ((\operatorname{id} \otimes_{k_{\mathfrak{p}}} F)(w^{\ast})) \\ &= v_F((\operatorname{id} \otimes_{k_{\mathfrak{p}}} F)\circ \alpha (w^{\ast}))  \\
        &= v_F ((\operatorname{id} \otimes_{k_{\mathfrak{p}}} F) (m^{\ast}))\neq 0.
    \end{align*}

    \textbf{Step 3:} \textit{We apply strong geometric reductivity using projective modules to construct~$d \in \mathbb{Z}_{+}$ and~$f\in S^d_{k}(W^{\ast})^{G_{k}}$ such that~$f(\pi(v_F)) \neq 0.$ Then we show its image~$g \in S^d_{k_{\mathfrak{p}}}(M^{\ast})^{G_{k_{\mathfrak{p}}}}$ satisfies~$g(v_F)\neq 0$.}

    Let~$\alpha: (W \otimes_k k_{\mathfrak{p}})^{\ast} \rightarrow M^{\ast}$ be the~$G$-module homomorphism as constructed in Step 2, and applying~$S^d(\cdot)$ induces a~$G$-module homomorphism~$S^d(\alpha)\colon S^d_{\loc{k}}((W \otimes_k \loc{k})^{\ast}) \longrightarrow S^d_{\loc{k}}(M^{\ast})$.
    Since~$G$ is strongly geometrically reductive over~$k$, there exists a positive integer~$d$ and~$f \in S^d_k(W^{\ast})^{G_k}$ such that~$f(\pi(v_F)) \neq 0$. By \cite[Lemma 2]{seshadri},
   ~$S^d_k(W^{\ast})^{G_k} \otimes_k k_{\mathfrak{p}} \cong S^d_{k_{\mathfrak{p}}} ((W \otimes_k k_{\mathfrak{p}})^{\ast})^{G_{k_{\mathfrak{p}}}}$, so~$f \in S^d_{k_{\mathfrak{p}}} ((W \otimes_k k_{\mathfrak{p}})^{\ast})^{G_{k_{\mathfrak{p}}}}$. Therefore,~$g \colonequals S^d(\alpha)(f) \in S^d_{\loc{k}}(M^{\ast})^{G_{\loc{k}}}$. Then we show that~$g(v_F)\neq 0$. As in the proof of Proposition \ref{PR_GR}, an element~$v_F \in (M \otimes_{\loc{k}} F)^{G_F}$ is a~$G$-equivariant map~$v_F \colon M^{\ast} \rightarrow k \rightarrow F$, so~$\pi(v_F) = v_F \circ \alpha$ is the map 
    \begin{equation*}
        \pi(v_F) \colon (W \otimes_k \loc{k})^{\ast} \overset{\alpha}{\longrightarrow} M^{\ast} \longrightarrow k \longrightarrow F,
    \end{equation*} which further induces
    \begin{equation*}
        S^d(\pi(v_F)) \colon S^d((W \otimes_k \loc{k})^{\ast}) \overset{S^d(\alpha)}{\longrightarrow} S^d(M^{\ast}) \longrightarrow S^d(k) \longrightarrow S^d(F).
    \end{equation*} Hence,
    \begin{equation*}
        g(v_F) = S^d(v_F)(g) = S^d(v_F)(S^d(\alpha)(f)) = S^d(v_F\circ \alpha)(f) = S^d(\pi(v_F))(f) = f(\pi(v_F)) \neq 0.
    \end{equation*}

    Hence, for any~$G_{k_{\mathfrak{p}}}$-module~$M$, finitely generated and projective over $k_{\mathfrak{p}}$, and~$v \in (M \otimes_{k} F)^{G_F}$, there exists a positive integer~$d$ and~$g \in S^d_{k_{\mathfrak{p}}}(M^{\ast})^{G_{k_{\mathfrak{p}}}}$ such that~$g(v) \neq 0$.
\end{proof}

Then we are able to prove the converse of Proposition \ref{PR_GR}.
\begin{cor}\label{GR_implies_PR}
    Let~$k$ be a regular ring of Krull dimension at most two. If~$G$ is geometrically reductive over~$k$, then~$G$ is power reductive over~$k$.
\end{cor}

\begin{proof}
    By Lemma \ref{strong_GR} and Theorem \ref{global_to_local},~$G$ is geometrically reductive over~$k_{\mathfrak{m}}$ for all maximal ideals~$\mathfrak{m}.$ Note that~$G_{k_{\mathfrak{m}}}$ is also smooth and has connected geometric fibres, so by Proposition~\ref{PR_GR},~$G_{k_{\mathfrak{m}}}$ is power reductive for all maximal ideals~$\mathfrak{m}$ of~$k$. By Section~3.1 of \cite{Franjou_VanderKallen},~$G$ is power reductive.
\end{proof}

\begin{cor}\label{GR_fibre}
    Let~$G$ be a smooth affine algebraic group scheme with connected geometric fibres over a Noetherian ring $k$. Then ~$G$ is geometrically reductive if and only if all geometric fibres of~$G$ are geometrically reductive. 
\end{cor}
\begin{proof}
    The forward direction is obvious from the definition. Conversely, if all geometric fibres of~$G$ are geometrically reductive, then all geometric fibres of~$G$ are reductive by \cite{Nagata_Miyata}. Therefore, ~$G$ is reductive and therefore geometrically reductive by \cite[Theorem 1]{seshadri}.
\end{proof}

\begin{cor}\label{PR_R}Let~$k$ be a commutative ring. Let~$G$ be a smooth affine algebraic group scheme over~$k$ with connected geometric fibres. 
    \begin{enumerate}[label =$(\alph*)$]
        \item~$G$ is reductive if and only if it is power reductive.
        \item~$G$ is power reductive if and only if all the geometric fibres of~$G$ are power reductive.
    \end{enumerate}
\end{cor}

\begin{proof} Part (a) is immediate from \cite[Theorem 5]{Franjou_VanderKallen} and \cite[Proposition 16]{PR_vdK}. Part (b) is immediate from \cite[Proposition 16]{PR_vdK} and  \cite{Nagata_Miyata}.
\end{proof}

\subsection{Relationship to Linearly Reductive Groups}\label{Sec_relation}

After we have studied the equivalence relations of reductive, power reductive and geometrically reductive groups over rings, we want to see how they relate to linearly reductive groups over rings. We start by looking at the special case of power reductive groups when~$d = 1.$

\begin{defn}
    Let~$G$ be a flat affine algebraic group scheme over a commutative ring~$k$. We say ~$G$ is \emph{power reductive of degree one} if~$d = 1$ in Definition \ref{PR_1}.
\end{defn}

\begin{prop}\label{PR_deg_one_property} Let~$G$ be a flat affine algebraic group scheme over a commutative ring.
    The following two statements are equivalent:
    \begin{enumerate}[label =$\left(\alph*\right)$]
        \item~$G$ is power reductive of degree one.
        \item~$G$ is linearly reductive.
    \end{enumerate}
\end{prop}

\begin{proof}
    $(a) \Rightarrow (b)$: Let~$V, W$ be~$G$-modules over~$k$ and let ~$\varphi: V \twoheadrightarrow W$ be a~$G$-module epimorphism. Take~$w \in W^G$. Let~$L$ be the cyclic module generated by~$w$, i.e.,~$L = k\cdot w.$ Then~$L \subseteq W.$ Let~$M:= \varphi^{-1}(L) \subseteq V$. By the definition of power reductive of degree one, there exists~$v \in M^G \subseteq V^G$ such that~$\varphi(v) = w$. That is,~$\varphi\mid_{V^G}: V^G \twoheadrightarrow W^G$ is surjective, whence ~$G$ is linearly reductive.

    $(b) \Leftarrow (a)$: Let~$\varphi: M \twoheadrightarrow L$, where~$L$ is a cyclic module on which~$G$ acts trivially. Then~$\varphi$ induces~$M^G \twoheadrightarrow L^G = L$. This shows that~$G$ is power reductive of degree one.
\end{proof}

\subsection{Summary}
Let~$G$ be a flat affine algerbaic group scheme over a commutative ring~$k$.
In the previous subsections, we have seen that when~$G$ is smooth and with connected geometric fibres, there are the following relations between different notions of reductivity: 

    
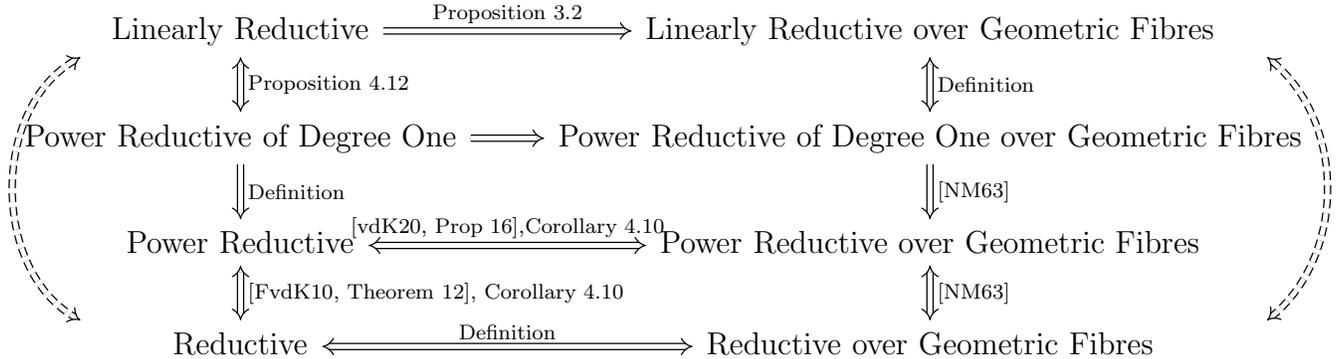
\begin{tikzcd}[trim left = -7.3cm]
        &\textnormal{Linearly Reductive} \arrow[ddd, Leftrightarrow, shift right = 15, bend right = 60, dashed]\arrow[r, Rightarrow, " \textnormal{Proposition } \ref{LRG_base_change}"] \arrow[d, Leftrightarrow, "\textnormal{Proposition \ref{PR_deg_one_property}}"]&\textnormal{Linearly Reductive over Geometric Fibres}\arrow[d, Leftrightarrow, "\textnormal{Definition}"]\arrow[ddd, Leftrightarrow, shift left = 40, bend left = 50, dashed]\\
        &\textnormal{Power Reductive of Degree One} \arrow[r, Rightarrow] \arrow[d, Rightarrow, "\textnormal{Definition}"]&\textnormal{Power Reductive of Degree One over Geometric Fibres}\arrow[d, Rightarrow, "\textnormal{\cite{Nagata_Miyata}}"]\\
        &\textnormal{Power Reductive} \arrow[r, Leftrightarrow, "\textnormal{\cite[Prop 16]{PR_vdK},Corollary \ref{PR_R}}"] \arrow[d, Leftrightarrow, "\textnormal{\cite[Theorem 12]{Franjou_VanderKallen}, Corollary \ref{PR_R}}"]&\textnormal{Power Reductive over Geometric Fibres}\arrow[d, Leftrightarrow, "\textnormal{\cite{Nagata_Miyata}}"]\\
        &\textnormal{Reductive} \arrow[r, Leftrightarrow, "\textnormal{Definition}"] &\textnormal{Reductive over Geometric Fibres}\\
    \end{tikzcd}

Note that we have the reverse arrows for the first two rows, if we assume that~$k$ is Noetherian \cite[Theorem 1.2(i)]{Margaux2009}. The dashed arrows are true when~$\operatorname{char}(k) = \operatorname{char}(\kappa(\mathfrak{p})) = 0,$ for all~$\mathfrak{p}$ prime ideals (See \cite{Nagata_Miyata}). If there exists a geometric fibre of positive characteristic, the only connected linearly reductive groups over~$k$ are tori. Indeed, the only connected linearly reductive groups over fields of positive characteristic are tori (See \cite{Nagata61}).

If we further assume that~$k$ is a regular ring of dimension at most two, we can add the following to the diagram:

    \begin{tikzcd}[trim left = -7.3cm]
        &\textnormal{Reductive} \arrow[r, Leftrightarrow, "\textnormal{Definition}"] \arrow[d, Leftrightarrow, "\textnormal{\cite[Theorem 1]{seshadri}} \textnormal{Corollary \ref{GR_implies_PR},\ref{PR_R}}"]&\textnormal{Reductive over Geometric Fibres}\arrow[d, Leftrightarrow, "\textnormal{\cite{Nagata_Miyata}}"]\\
        &\textnormal{Geometrically Reductive} \arrow[r, Leftrightarrow, "\textnormal{Cor \ref{GR_fibre}}"] &\textnormal{Geometrically Reductive over Geometric Fibres}\\
    \end{tikzcd}

\subsection*{Author Information} \hfill

\bigskip

\noindent Yidi Wang

\noindent Department of Mathematics, University of Western Ontario, ON, Canada. 

\noindent Email: ywan6443@uwo.ca

\end{document}